\documentclass{amsart}

\usepackage[vcentermath,enableskew,stdtext]{youngtab}
\usepackage{youngtab}
\usepackage{amsfonts}
\usepackage{amsmath}
\usepackage{amscd}

\usepackage[all]{xy}

\newtheorem{dummy}{dummy}[section]
\newtheorem{lemma}[dummy]{Lemma}
\newtheorem{theorem}[dummy]{Theorem}
\newtheorem{corollary}[dummy]{Corollary}
\newtheorem{proposition}[dummy]{Proposition}
\theoremstyle{definition}
\newtheorem{definition}[dummy]{Definition}
\newtheorem{example}[dummy]{Example}
\newtheorem{remark}[dummy]{Remark}

\newcommand{\Aff}[0]{\mathbb A}
\newcommand{\Affn}[0]{{\mathbb A}^{n^2} - \{0\}}
\newcommand{\DMeff}{DM^{eff}_-(F)}
\newcommand{\DMeffXX}{DM^{eff}_-(\XX)}
\newcommand{\End}{\mathcal End}
\newcommand{\Filt}[1]{\nu_\XX^{\ge #1}}
\newcommand{\GL}{\mathbf {GL}}
\newcommand{\Gm}{\mathbb G_m}
\newcommand{\Gr}{\mathbf {Gr}}
\newcommand{\Hom}{\mathcal Hom}
\newcommand{\JJ}{\mathcal J}
\newcommand{\MSBD}[1]{M(SB(A))\{#1\}}
\newcommand{\OO}{\mathcal O}
\newcommand{\Piece}[1]{\nu_\XX^{#1}}
\newcommand{\Proj}[0]{\mathbb P}

\newcommand{\quat}{\binom{a,b}{F}}
\newcommand{\QQ}[0]{\mathcal Q}
\newcommand{\Tate}[0]{\mathbb Z}
\newcommand{\Tatepure}[1]{\mathbb Z\{#1\}}
\newcommand{\Tatemixed}[2]{\mathbb Z(#1)[#2]}
\newcommand{\XX}[0]{\mathcal X}
\newcommand{\ZZ}[0]{\mathbb Z}

\title{On motives of algebraic groups associated to a divsion algebra}
\author{Evgeny Shinder}

\begin{document}

\maketitle


\section{Introduction}

In this paper we consider motives and motivic cohomology of algebraic groups $GL_1(A)$ and $SL_1(A)$ for a central simple
algebra $A$ of prime degree $n$ over a field $F$. Motivation to study these groups comes from the problems arising
in algebraic $K$-theory, in particular non-triviality of $SK_1(A)$ \cite{SusSK1}, \cite{MerSK1}.

It is proved by Biglari \cite{Big} that motives of {\it split} reductive algebraic groups such as $GL_n(F)$ and $SL_n(F)$ are Tate
motives. Furthermore, using higher Chern classes in motivic cohomology constructed by Pushin \cite{Pu} 
one can write down explicit direct sum decompositions for the motives
of these two groups with integral coefficients. 
Non-split algebraic groups such as $GL_1(A)$ and $SL_1(A)$ are more intricate.
We note however that all the complications lie in $n$-torsion effects ($n$ is the degree of $A$): the situation becomes 
trivial if we make $n$ invertible.   

For $GL_1(A)$ we follow an idea of Suslin to split the motive $M(GL_1(A))$ into two pieces: the first piece is a very simple Tate motive,
whereas the second piece is a twisted Tate motive $M$ over $\XX$, where $\XX$ is the Voevodsky-Chech simplicial scheme
associated to the Severi-Brauer variety $SB(A)$ (Theorem \ref{MotiveM}). We investigate the structure of the latter motive $M$ using the twisted
slice filtration, and compute the second differential in the arising spectral sequence (Theorem \ref{Differential}).
Using the spectral sequence we compute some lower weight motivic cohomology groups of $GL_1(A)$ (Corollary \ref{SmallWeights}).
We also consider the case of degree 2 algebra where one can write explicit decomposition for $M(GL_1(A))$
(Proposition \ref{MGLquat}). 

For $SL_1(A)$ we only consider algebras $A$ of degree either 2 or 3. 
In both of these cases $SL_1(A)$ admits an explicit smooth compactification $X_A$
given as a hyperplane section of a generalized Severi-Brauer variety (Proposition \ref{Compactif}).
In the degree 2 case $X_A$ is a 3-dimensional quadric, whose motive can be computed explicitly
(Proposition \ref{MSLquat}).
In the degree 3 case $X_A$ is a hyperplane section of the twisted form of the Grassmannian $Gr(3,6)$.
The motives of such hyperplane sections with coefficients in $\ZZ/3$ were already considered in \cite{Sem}.
We give a slightly different proof of the decomposition we need with integral (or more precisely, with 
$\ZZ[\frac{1}{2}]$) coefficients
(Proposition \ref{MotDecomp}), using a version of Rost nilpotence theorem that
we prove along the way (Proposition \ref{RostNilp}).

The author expresses his gratitute to A.Suslin for numerous conversations on matters discussed
in the paper.      

We fix the notation we need.

$\bullet$ $F$ is a perfect field. We assume $char(F) \ne 2$ whenever we speak of quaternion algebars and 
$char(F)=0$ in Section 5.3.

$\bullet$ $A$ is a central simple algebra over $F$ of degree $n$. We assume $n$ to be prime in Section 4.

$\bullet$ $SB_k(A), \, k \ge 1$, is the generalized Severi-Brauer variety introduced in Section 2.
In particular $SB_1(A) = SB(A)$ is the usual Severi-Brauer variety.

$\bullet$ $\DMeff$ is the Voevodsky triangulated category of motives \cite{V0}, with the Tate twist $(*)$ and
shift $[*]$. We also use $\{*\} = (*)[2*]$, especially when working with motives of smooth projective
varieties. For example, 
$M(\Proj^1) = \Tate \oplus \Tatemixed{1}{2} = \Tate \oplus \Tatepure{1}$.

$\bullet$ If $\JJ$ is a vector bundle we write $\JJ^*$ for the dual bundle. 

\section{Varieties associated to central simple algebras}

In this section $A$ is a central simple algebra $A$ over a field $F$ of degree $n$,
i.e. an associative unital algebra of dimenstion $n^2$ over $F$ that
has no nontrivial two-sided ideals and such that the center of $A$ coinsides with $F$.

According to the Wedderburn theoreom, $A$ is isomorphic to the matrix algebra $M_n(D)$
over a central division algebra $D$ over $F$. $A$ is called split if it is isomorphic to $M_n(F)$.
It is well known that any central simple algbra splits in some finite separable extension of scalars $E/F$:
$$
A_E = A \otimes_F E \cong M_n(E).
$$

By the standard Galois descent arguments the set of isomorphism classes of central simple algebras over $F$ is
in bijection with $H^1(Gal(F^{sep}/F), PGL_n(F^{sep})$.
Galois descent also implies that $det: M_n(F^{sep}) \to F^{sep}$ and $tr: M_n(F^{sep}) \to F^{sep}$
descend to define the so called reduced norm map $Nrd: A \to F$ and
the reduced trace map $Trd: A \to F$.

\begin{example}
A quaternion algebra $\quat$ is defined for $a,b \in F^*$ to be
a vector space of dimension $4$ with the basis $1, i, j, k$ and multiplication
$i^2 = a, j^2 = b, ij = -ji = k$
(under the assumption $char(F) \ne 2$). It follows from the Wedderburn theorem, that
$\quat$ either splits or is a division algebra.
$Trd$ and $Nrd$ are
the usual trace and norm: $Trd(x + yi + zj + wk) = 2x$, $Nrd(x + yi + zj + wk) = x^2 - ay^2 - bz^2 + abw^2$.
\end{example}

\subsection{Generalized Severi-Brauer varieties}

\label{GenSB}

The generalized Severi-Brauer variety $SB_k(A)$ is a closed subvariety in $Gr(kn, A)$
representing the functor which associates to a commutive algebra
$R$ over $F$ the set
$$
SB_k(A)(R) = \{right \hskip0.1cm ideals \hskip0.1cm of \hskip0.1cm A \otimes R \hskip0.1cm
which \hskip0.1cm are \hskip0.1cm projective \hskip0.1cm of \hskip0.1cm rank \hskip0.1cm kn \hskip0.1cm over \hskip0.1cm R\}
$$

\begin{remark}
\label{Istar}
1. Let $V$ be a vector space of dimension $n$ over $F$, and let $A$ be a split central simple algebra $A = End(V)$.
In this case right ideals in $A$ of rank $kn$ have the form $V^* \otimes U \subset V^* \otimes V \cong End(V)$, where
$U$ is a subspace in $V$ of dimension $k$. Therefore we have a canonical identification
$$
SB_k(End(V)) = Gr(k, V).
$$

2. If $A$ has a right ideal $I$ of rank $n$, $A$ is split.
Indeed,
the right multiplication action $R_\alpha: I \to I, a \in A$ satisfies $R_{\alpha \beta} = R_\beta R_\alpha$,
and the homomorphism
$$
R: A \to End(I)^{op} = End(I^*)
$$
is an isomorphism by the Schur lemma.
\end{remark}

In general $SB_k(A)$ is a form of $Gr(k,V) = Gr(k, n)$ twisted by the cocycle defining $A$
and the usual Severi-Brauer variety $SB(A) = SB_1(A)$ is a twisted form of the projective
space $\Proj(V) \cong \Proj^{n-1}$.

\begin{example}
In the case $A = \quat$, $SB(A)$ is a conic in $\Proj^2$ defined by the equation
$x^2 = ay^2 + bz^2$.
\end{example}

By definition, $SB_k(A)$ is endowed with a locally free sheaf $\JJ_k$ 
of right $A$-modules of rank $k$.
$\JJ_k$ is a subsheaf of $\OO_{SB_k(A)} \otimes A$.

\begin{remark}
\label{Jsplit}
In the split case $A = End(V)$, $\JJ_k$ is identified
with $V^* \otimes \xi = \Hom(V, \xi)$ over $Gr(k,V)$,
where $\xi$ denotes the tautological rank $k$ bundle on $Gr(k,V)$.
\end{remark}

We will write $\JJ$ for $\JJ_1$ over $SB(A)$. $\JJ^*$ denotes the dual 
sheaf.

\begin{lemma}
\label{Jstar}
The sheaf of algebras $\OO_{SB(A)} \otimes A$ is isomorphic to $\End(\JJ^*)$.
\end{lemma}
\begin{proof}
The isomorphism is given by the right action of $A$ on $\JJ$, as in \ref{Istar}, 2.
\end{proof}

If $p: E \to T$ is a vector bundle we will write
$\Gr_{T}(k, E)$ for the
Grassmannian bundle over $T$ of $k$-planes in $E$.
$Gr_T(k,E)$ comes equipped with a short exact sequence of vector bundles:
$$
0 \to \xi_E \to p^*E \to \QQ_E \ \to 0.
$$
If $E$ is a trivial bundle over $T = Spec(F)$, 
we use the notation $\xi = \xi_E$ and $\QQ = \QQ_E$.

\begin{proposition}
\label{SplitSB}
There is a canonical isomorphism of
varieties over $SB(A)$
$$
SB(A) \times SB_k(M_l(A)) \cong \Gr_{SB(A)}(k, \JJ^{*\oplus l}),
$$
where $\JJ$ is the tautological sheaf of ideals on $SB(A)$.

Furthermore, the tautological bundles $\xi_{\JJ^{*\oplus l}}$ and $\QQ_{\JJ^{*\oplus l}}$ over 
$\Gr_{SB(A)}(k, \JJ^{*\oplus l})$
correspond under this isomorphism to vector bundles
over $SB(A) \times SB_k(M_l(A))$, which in the split case become identified with
$p_1^*(\OO(1)) \otimes p_2^*(\xi)$ and $p_1^*(\OO(1)) \otimes p_2^*(\QQ)$ respectively.
\end{proposition}

\begin{proof}
The first assertion follows from Lemma \ref{Jstar}.
To prove the second assertion, consider the split case $A = End(V)$, so we can
identify $J^*$ with $V \otimes \OO(1)$ by \ref{Jsplit}.
Then the isomorphism in
question becomes the canonical identification:
$$
\Gr_{\Proj(V)}(k, V^{\oplus l} \otimes \OO) \cong \Gr_{\Proj(V)}(k, V^{\oplus l} \otimes \OO(1)),
$$
and the claim follows from the lemma below.

\end{proof}

\medskip

\begin{lemma}
Let $E$ be a vector bundle and $L$ be a line
bundle over the same base $T$. Then
the tautological bundles
$\xi_{E \otimes L}$ and $\QQ_{E \otimes L}$ over
$\mathbf{Gr}_T(k, E \otimes L)$ correspond under the canonical
isomorphism  
$$
\mathbf{Gr}_T(k, E) \cong \Gr_T(k, E \otimes L)
$$
of varieties over $T$ to
$p^* L \otimes \xi_E$ and $p^* L \otimes \QQ_E$ respectively
($p$ is the projection to $T$).
\end{lemma}
\begin{proof}

Consider the pair $S = (\xi, Gr(k,n))$.
$S$ admits a $GL_n(F)$ action
defined in an obvious way:
$$
(g \in GL_n(F), v \in V \subset F^n) \to (gv \in gV \subset F^n).
$$

Therefore for any cocycle $\theta \in H^1_{Zar}(T, GL_n(F))$ we
can construct a Zarisky locally trivial family $S_\theta$.
In fact, if $\theta = \theta_E$ is the cocycle defining $E$,
then $$S_{\theta_E} \cong (\xi_E, \Gr_T(k, E)).$$

Now $\theta_{E\otimes L} = \theta_E \theta_L$, $\theta_L \in H^1_{Zar}(T, \Gm)$.
$\Gm$ acts on the trivial pair $S$ by:
$$
(\lambda \in \Gm(F), v \in V \subset F^n) \to (\lambda v \in V \subset F^n).
$$

Therefore the pair $S_{\theta_{E \otimes L}}$ is isomorphic to
$(p^*(L) \otimes \xi_E, \Gr_T(k, E))$.

The same argument applies to $\QQ_E$.

\end{proof}

\medskip

Consider now the case $SB_n(M_l(A))$, where $A$ is a central simple algebra of degree $n$.
$SB_n(M_l(A))$ is a twisted form of the
Grassmannian $Gr(n, V^{\oplus l}) = Gr(n,ln)$.

We will need the following definition.

\begin{definition}
$\alpha_1, \alpha_2, \cdots \alpha_l \in A$ are called independent if there does not exist
$0 \ne \lambda \in A$ with the property $\alpha_1 \lambda = \alpha_2 \lambda = \cdots
\alpha_l \lambda = 0.$
\end{definition}

It is easy to see that
$\alpha_1, \alpha_2, \cdots \alpha_l \in A$ are independent if and only if the {\it left} ideal
$A \alpha_1 + A \alpha_2 + \cdots + A\alpha_l \subset A$ equals $A$.

\begin{remark}
If $A$ is a division algebra, then any $\alpha_1, \alpha_2, \cdots \alpha_l \in A$ are
independent unless they are all zero. On the other hand, if $A$ is a split algebra $A = End(V)$,
then $\alpha_1, \alpha_2, \cdots \alpha_l \in A$ are independent if and only if
$$(\alpha_1, \alpha_2, \cdots, \alpha_l): V \to V^{\oplus l}$$
is injective.
\end{remark}

\begin{lemma} Let $A$ be a central simple algebra of rank $n$ over a field $F$.

1) $SB_n(M_l(A))$ represents the functor which associates
to a commutive $F$-algebra $R$ the set of nondegenerate columns
modulo the right action of $A_R$: 
$$SB_n(M_l(A))(R) = \{(\alpha_1, \alpha_2, \cdots, \alpha_l):
\alpha_1, \alpha_2, \cdots, \alpha_l \in A_R \, independent\} / A_R^*.$$

2) Vector bundle $\xi_n^{\oplus n}$ and line bundle $\Lambda^n \xi_n$ over $Gr(n,V^{\oplus l})$ canonically
descend to $SB_n(M_l(A))$.

\end{lemma}

\begin{proof}
1) A right ideal $R$ of dimension $ln^2$ in $M_l(A)$ splits via the action
of the idempotents $e_i = e_{i,i} \in M_l(A)$ as a direct sum
of $l$ isomorphic right $A$-submodules $R_0 \subset A^{\oplus l}$ of dimension $n^2$, corresponding
to the columns of $R$.

Any $A$-module is a direct sum of minimal $A$ modules, and comparing the dimensions we
see that in fact $R_0$ is isomorphic to the trivial right $A$-module $A$.
Let $\phi: A \to R_0$ be an isomorphism, and let $(\alpha_1, \alpha_2, \cdots, \alpha_l) = \phi(1)$.
By definition $\alpha_1, \alpha_2, \cdots, \alpha_l$ are independent
and any two vectors $v_1, v_2 \in A^{\oplus l}$
consisting of independent entries generate the same right submodule if
and only if $v_1 = v_2 \lambda$ for some $\lambda \in A^*$.

2) The same argument as above applied to the tautological sheaf of ideals $\JJ_n$ on $SB_n(M_l(A))$ shows that
this sheaf splits as a direct sum of $n$ copies of $A$-module sheaf, which becomes isomorphic
to $\xi_n^{\oplus n}$ if the algebra splits. $\Lambda^n \xi_n$
is the line bundle associated with the sheaf of $A$-modules $\xi_n^{\oplus n}$ via the character
$Nrd: GL_1(A) \to \Gm$
\end{proof}

\begin{remark}
The usual description of the functor of points for $Gr(n,ln)$ amounts to choosing $n$ linearly
independent columns of size $ln$, forming an $ln \times n$ matrix of rank $n$,
which is considered modulo the right action of $GL_n(F)$.
The lemma says that for $SB_n(M_l(A))$ the $n \times n$ blocks of this matrix are rational.
\end{remark}

\begin{remark}
\label{Plucker}
The second statement of the lemma implies that there is a twisted Pl\"{u}cker embedding
$$
SB_n(M_l(A)) \to \Proj^{\binom{nl}{l}}.
$$

To make this map expicit, we pick a basis of the global sections of $\Lambda^n \xi_n$
consisting of the forms
$H_{\beta_1, \cdots, \beta_l}(\alpha_1, \cdots, \alpha_l) = Nrd(\sum_{i} \beta_i \alpha_i)$.

For example, in the case of a quaternion algebra $A$ and $l=2$ we can pick the basis to be
$Nrd(\alpha_1), Nrd(\alpha_2), Nrd(\alpha_1+\alpha_2), Nrd(\alpha_1+i\alpha_2),
Nrd(\alpha_1+j\alpha_2),
Nrd(\alpha_1+k\alpha_2)$.

\end{remark}

\medskip

\medskip

\subsection{Algebraic groups associated to $A$}

For a central simple algebra $A$ over $F$, one can define linear algebraic groups $GL_1(A)$, $SL_1(A)$.
For any $R$ is a commutative algebra over $F$ the $R$-points of these groups are:

$$GL_1(A)(R) = (A \otimes_F R)^* = \{g \in A \otimes_F R: Nrd(g) \ne 0\}$$
$$SL_1(A)(R) = \{g \in (A \otimes_F R)^*: Nrd(g) = 1\}.$$

Note that $SL_1(A) = Ker(Nrd: GL_1(A) \to \Gm)$ in the category of algebraic groups.
$GL_1(A)$ and $SL_1(A)$ are forms of $GL_n(F)$ and $SL_n(F)$ respectively
twisted by the cocycle defining $A$.

\begin{example}
For the quaternion algebra $A = \quat$,
$GL_1(A)$ is an open subscheme in $\Aff^4$ defined by
$x^2 - ay^2 - bz^2 + abw^2 \ne 0$,
and $SL_1(A)$ is a quadric in $\Aff^4$ defined by
$x^2 - ay^2 - bz^2 + abw^2 = 1$.
\end{example}

Let $E \to T$ be a vector bundle of rank $n$ and consider the associated group scheme
$\GL_T(E)$ of local automorphisms of $E$ over $T$.
Let $\alpha_E$ be the tautological automorphism of $p^*(E)$ where
$p: \GL_T(E) \to T$ is the projection.
Via explicit description of $K_1$ by Gillet and Grayson \cite{GG}, $\alpha_E$
defines an element $[\alpha_E] \in K_1(\GL_T(E))$.

This applies in particular to the case of the trivial bundle $E = V$ over a point,
in which case we denote the corresponding element in $K_1(GL_n(F))$ by $[\alpha_0]$.

The following proposition is analogous to \ref{SplitSB}.

\begin{proposition}
\label{SplitGL1}
There is a canonical isomorphism of varieties over $SB(A)$
$$
SB(A) \times GL_1(A) \cong \GL_{SB(A)}(J^*),
$$
where $\JJ$ is the tautological sheaf of ideals on $SB(A)$.

Furthermore, the tautological class $[\alpha_{J^*}] \in K_1(\GL(J^*))$
corresponds under this isomorphism to a class in $K_1(SB(A) \times GL_1(A))$
which in the split case is identified with $[p_1^*(\OO(1))] \cdot [p_2^*(\alpha_0)]$
where the product is the usual product $K_0 \otimes K_1 \to K_1$.
\end{proposition}
\begin{proof}

The first assertion follows from Lemma \ref{Jstar}.
To prove the second assertion, consider the split case $A = End(V)$, and we
identify $J^*$ with $V \otimes \OO(1)$ by \ref{Jsplit}.
Then the isomorphism in
question becomes the canonical identification:
$$
\Proj(V) \times GL_1(End(V)) = \GL_{\Proj(V)}(V \otimes \OO) =  \GL_{\Proj(V)}(V \otimes \OO(1)),
$$
and the claim follows from the following lemma.
\end{proof}

\begin{lemma}
Let $E$ be a vector bundle and $L$ be a line bundle over the same base $T$, which is assumed to be quasiprojective.
Then the tautological class $[\alpha_{E \otimes L}] \in K_1(\GL_T(E \otimes L))$ corresponds
to $[p^* L] \cdot [\alpha_E] \in K_1(\GL_T(E))$ under the canonical isomorphism
of group schemes over $T$ ($p$ is the projection to $T$)
$$
\GL_T(E) \cong \GL_T(E \otimes L).
$$
\end{lemma}
\begin{proof}
Using the Jouanolu trick, we assume that $T = Spec(R)$.
We will use the same letters $E$ and $L$ for the $R$-projective modules
corresponding to bundles $E$ and $L$.

$GL(E)$ is an affine scheme.
Let $S = \Gamma(GL(E), \OO_{GL(E)})$.
$\alpha_E$ is the tautological automorphism of the module $p^* E$,
and $\alpha_{E \otimes L} = \alpha_E \otimes id_{p^* L}$ is the automorphism of
$p^*(E \otimes L) = p^*(E) \otimes p^*(L)$.

It follows from the definition of the product in the K-theory of rings
$K_0(S) \otimes K_1(S) \to K_1(S)$,
that $[p^* L] \cdot [\alpha_E]$ = $[\alpha_{E \otimes L}].$

\end{proof}

\section{The slice filtration}

We work in the category $\DMeff$ of motivic complexes defined by Voevodsky \cite{V0}.
Recall that $\DMeff$ is a tensor triangulated category which admits a covariant monoidal functor
$$M: Sm/F \to \DMeff,$$
satisfying the usual properties such as Mayer-Vietoris and localization distinguished
triangles.

For any smooth variety $X$ there is a splitting
$$
M(X) = \widetilde{M}(X) \oplus \Tate,
$$
where $\widetilde{M}(X)$ is defined to be the cone of the canonical
morphism $M(X) \to M(pt) = \Tate$.

Motivic cohomology groups 
of degree $p \in \ZZ$ and weight $q \ge 0$)
are defined to be
$$
H^{p,q}(X) := Hom_{\DMeff}(M(X), \Tate(q)[p]),
$$
so that distinguished triangles in $\DMeff$ become long exact sequences in motivic cohomology of each weight.

We also use reduced motivic cohomology groups
$$
\widetilde{H}^{p,q}(X) := Hom_{\DMeff}(\widetilde{M}(X), \Tate(q)[p]).
$$

Consider the Cech simplicial scheme $\XX = \check{C}(SB(A)$ \cite{V}.
$\XX$ is defined such that $\XX_k = SB(A)^{k+1}$ and the face and degeneracy maps
are taken to be partial projections and diagonals. 
The canonical morphism $M(\XX) \to \Tate$ is an isomorphism if $SB(A)$ has
an $F$-point (i.e. if algebra $A$ splits).

We introduce a tensor triangulated category $\DMeffXX$ of motives
over $\XX$ as the full subcategory of $\DMeff$,
consisting of objects $M$ satisfying the property that
the canonical morphism
$$M \otimes M(\XX) \to M \otimes \Tate = M$$ 
is an isomorphism \cite{V1}.
Note that $\XX$ is an {\it embedded} simplicial scheme, which means
that $M(\XX) \otimes M(\XX) = M(\XX)$ and so $M(\XX)$ is an object
in $DM^{eff}_\XX$. We will write $\Tate_\XX$ for $M(\XX)$.

The full embedding $\DMeffXX \subset \DMeff$ admits a right adjoint
functor 
$$\Phi: \DMeff \to \DMeffXX,$$
which on objects is defined to be
$$\Phi(M) = M \otimes M(\XX).$$

\begin{remark}
\label{CohomAdj}
It follows from the adjunction property that
for any motive $M$ in $\DMeffXX$, $q \ge 0$, $p \in \ZZ$
$$H^{p,q}(M, \ZZ) = Hom_{\DMeff}(M, \Tatemixed{q}{p}) \cong
Hom_{\DMeffXX}(M, \Tate_\XX(q)[p]).$$    
\end{remark}

Let $DT(\XX) \subset \DMeffXX$ denote the subcategory of effective Tate motives over $\XX$.

We consider the {\it slice filtration} on these categories as defined in \cite{V1} (see also \cite{HK}) for any object $M$ of $\DMeffXX$. For each $q \ge 0$ the $q$-th term of the slice
filtration is given by:

$$
\Filt{q} M = \underline{Hom}_{\DMeff)}(\Tate(q), M)(q) \otimes \Tate_\XX.
$$

The internal $Hom$-object above exists by \cite{V0}, Proposition 3.2.8.

\begin{remark}
It is easy to see using the adjunction property that
$$
\underline{Hom}_{\DMeff}(\Tate(q), M)(q) \otimes M(\XX)
$$
is in fact isomorphic to
$$
\underline{Hom}_{\DMeffXX}(\Tate_\XX(q), M). 
$$
On the other hand, if $M$ is an object in $DT(\XX)$, then
$\underline{Hom}_{\DMeffXX}(\Tate_\XX(q), M)$ is isomorphic to
$\underline{Hom}_{DT(\XX)}(\Tate_\XX(q), M)$,
so that for Tate motives our slice filtration coincides with
the one from \cite{V1}.

\end{remark}

We define $\Piece{q}$ as the cone in the distinguished triangle
$$
\Filt{q+1}(M) \to \Filt{q}(M) \to \Piece{q}(M) \to \Filt{q+1}(M)[1].
$$

One can see that the slice filtration $\{\Filt{q}\}$ is functorial, commutes
with extension of scalars and that 
for each $k,j \ge 0$ and $l \in \ZZ$ satisfies
$$
\Filt{k+j} (M(j)[l]) = \Filt{k}(M)(j)[l]
$$
and
$$
\Filt{k}(M \oplus M') = \Filt{k}(M) \oplus \Filt{k}(M'). 
$$

\begin{remark}
\label{FiltSplitTate}
For a split Tate motive 
$M = \oplus_{p,q} \Tate_\XX(p)[q]^{\oplus a_{p,q}}$
we have 
$$\Filt{k}(M) = \oplus_{p \ge k,q} \Tate_\XX(p)[q]^{\oplus a_{p,q}}$$
and
$$\Piece{k}(M) = \oplus_{q} \Tate_\XX(k)[q]^{\oplus a_{k,q}}.$$
\end{remark}

The following two lemmas provide examples of motives lying in $\DMeffXX$ and $DT(\XX)$ respectively.

\begin{lemma}
\label{XXInvar}
Let $T$ be a variety over $F$.

1. If $T$ is smooth and for each generic point $\eta$ of $T$ $A_{F(\eta)}$ is a split
algbera then $M(T)$ lies in $\DMeffXX$.

2. Let $T \subset S$ be a closed embedding of $T$ into a smooth variety $S$.
If for each scheme-theoretic point $z \in T \; A_{F(z)}$ is a split algebra
then $M_T(S)$ lies in $\DMeffXX$.

\end{lemma}

\begin{proof}
To prove the first statement, we need to show that 
$M(T) \otimes C = 0$ where $C = cone(M(\XX) \to \Tate)$.
This follows from \cite{V}, Lemma 4.5.
To prove the second statement, we filter $T$ by closed subvarieties 
$$T_N \subset T_{N-1} \subset \dots \subset T_{1} \subset T_0 = T \subset S$$
where $T_k \backslash T_{k+1}$ are nonsingular.
We prove by the descending induction on $k$ that $M_{T_k}(S)$ is an object in $\DMeffXX$.
The base case $k = N$ follows from the first statement of the Lemma.
For the induction step, we use the distinguished triangle in $\DMeff$
$$
M_{T_k \backslash T_{k+1}}(S \backslash T_{k+1}) \to M_{T_k}(S) \to M_{T_{k+1}}(S)
\to M_{T_k \backslash T_{k+1}}(S \backslash T_{k+1})[1]
$$
of the triple
$$
(S \backslash T_{k+1}, S \backslash T_{k}) \subset (S, S \backslash T_{k}) 
\subset (S, S \backslash T_{k+1}). 
$$
Since by induction hypothesis and by applying the first part of the Lemma again,
$M_{T_{k+1}}(S)$ and $M_{T_k \backslash T_{k+1}}(S \backslash T_{k+1})$ lie in $\DMeffXX$, $M_{T_k}(S)$ also lies in $\DMeffXX$.

\end{proof}

\begin{lemma}
\label{TwistedTate}
Let $M$ be an object in $\DMeffXX$. Assume that $M_{F(SB(A))}$ is a split Tate motive 
of the form $\oplus_{p,q} \Tate(p)[q]^{\oplus a_{p,q}}$.
Then the slice filtration of $M$ in $DM_\XX$ has successive cones which are split Tate motives
$$
\Piece{p}(M) = \oplus_{q} \Tate_{\XX}(p)[q]^{\oplus a_{p,q}}.
$$
In particular, $M$ is a mixed Tate motive over $\XX$.
\end{lemma}

For the proof we need the following lemma, which we borrow
from \cite{Sus}.

\begin{lemma}
\label{CohomXX}
For any $M$ from $\DMeff$ and $p \in \ZZ$
the extension of scalars
$H^{p,0}(M) \to H^{p,0}(M_{F(SB(A))})$ 
is an isomorphism.
\end{lemma}

\begin{proof}
It is sufficient to prove the statement in the case $M = M(S)[j]$
where $S$ is a smooth connected scheme over $F$.
In this case the homomorphism in question takes the form:
$$H^{p-j,0}(S) \to H^{p-j,0}(S_{F(SB(A))}),$$
and both groups are equal $0$ for $p \ne j$.

$S$ is connected, and $SB(A)$ being geometrically irreducible
has separably generated function field $F(SB(A))$, hence
$S_{F(SB(A))}$ is connected as well.
Therefore if $p=j$ both cohomology groups in question
are isomorphic to $\ZZ$ with the map being the identity.

\end{proof}

Now we can prove Lemma \ref{TwistedTate}.
\begin{proof}
Let $\Piece{p} M = c_p(M)(p)$. Then
\begin{align*}
Hom(\Piece{p} M, \Tate_\XX(p)[q]) \\
= Hom(c_p(M), \Tate_{\XX}[q]) & \; \text{by the cancellation theorem} \\
=  H^{q,0}(c_p(M), \ZZ) & \; \text{by Remark \ref{CohomAdj}} \\
= H^{q,0}(c_p(M_{F(SB(A))}), \ZZ) & \; \text{by Lemma \ref{CohomXX}} \\
= H^{q,0}(\oplus_{r} \Tate[r]^{\oplus a_{p,r}}, \ZZ) & \; \text{by Remark \ref{FiltSplitTate}} \\
= \Tate^{\oplus a_{p,q}}.
\end{align*}

Therefore there exists a morphism 
$\phi_p: \Piece{p} M \to \oplus_{q} \Tate_{\XX}(p)[q]^{\oplus a_{p,q}}$ 
such that $\phi_p$ becomes an isomorphism after scalar extension to $F(SB(A))$.
This implies that $cone(\phi_p)_{F(SB(A))} = 0$, so that $cone(\phi_p) = cone(\phi_p) \otimes M(\XX) = 0$ by \cite{V}, Lemma 4.5, and thus $\phi_p$ is an isomorphism.
\end{proof}

\begin{remark}
As the example of $M = M(SB(A))$ shows, $M$ itself is not always a {\it split} Tate motive.
Indeed it is a result of Karpenko \cite{Kar} that for a division algebra
$A$, $M(SB(A))$ is indecomposable.
\end{remark}

\begin{example}
Let $A = \quat$, and let $M_{a,b} = M(SB(A))$ be the Rost motive.
In this case the slice filtration is the  distinguished triangle
$$\Tate_\XX(1)[2] \to M_{a,b} \to \Tate_\XX \to \Tate_\XX(1)[3]$$
from \cite{V}, Theorem 4.4. 
\end{example}

In section 5 we will need the following version of the Rost nilpotence theorem (cf \cite{CGM}, Cor. 8.4 and \cite{R}, Cor. 10),
which is a corollorary of Lemma \ref{CohomXX} and the existence of the slice filtration:

\medskip

\begin{proposition}
\label{RostNilp}
 Let $M = \bigoplus_{k=0}^n \mathbb
Z\{i_k\}$ be a pure Tate motive. Let $f: M(SB(A)) \otimes M \to
M(SB(A)) \otimes M$ be a morphism of motives. If $f_{F(SB(A))}$ is an
isomorphism then $f$ is an isomorphism.
\end{proposition}

\begin{proof}
Consider the slice filtration on $M(SB(A)) \otimes M$. By Lemma \ref{TwistedTate}
the slices $\Piece{p}(M(SB(A)) \otimes M)$ are equal to $\Tate_\XX\{p\}^{\oplus a_p}$,
for some $a_p \ge 0$.  
The morphisms induced on the slices 
are given by matrices with coefficients in $Hom(\Tate_\XX, \Tate_\XX)$,
and this group is identified with $\ZZ$ after extension of scalars to $F(SB(A))$
by Lemma \ref{CohomXX}.

\end{proof}

\medskip

\medskip

The slice filtration gives rise to an exact couple for each weight $j$
$$
E^{p,q} = H^{p+q}(\Piece{q}(M), \Tate(j)),
$$
$$
D^{p,q} = H^{p+q}(M, \Filt{q+1}(M), \Tate(j)),
$$
$$
\dots \to D^{p+1,q-1} \to D^{p,q} \to E^{p,q} \to D^{p+2,q-1} \to \dots
$$
and the corresponding spectral sequence
$$
E_2^{p,q} = H^{p+q}(\Piece{q}(M), \Tate(j)) \Rightarrow H^{p+q}(M, \Tate(j)),
$$
with the differential $d_2: H^{p+q-1}(\Piece{q+1} M, \Tate(j)) \to H^{p+q}(\Piece{q}M, \Tate(j))$
induced by a composition of morphisms forming the slice filtration:
$$
\Piece{q}(M) \to \Filt{q+1}(M)[1] \to \Piece{q+1}(M)[1].
$$


\section{The case of $GL_1(A)$}

\subsection{The split case}

We consider the group variety $GL_n(F)$ over a field $F$.
To give an explicit description of $M(GL_n(F))$ we use the
higher Chern classes $c_{1,i}$ for motivic cohomology
$$
c_i = c_{1,i}: K_1(X) \to H^{2i-1,i}(X), \; i \ge 1.
$$
as defined by Pushin \cite{Pu}. The classes $c_i$ are functorial, additive, and
admit the following product formula: if $[L] \in K_0(X)$ represents a class of a line bundle $L$, with $\lambda = c_1(L)$ and $[\alpha] \in K_1(X)$, then for
$[L] \cdot [\alpha]$ we have
$$
c_k([L] \cdot [\alpha]) = \sum_{i=1}^{k-1} (-1)^i \binom{k-1}{i} \lambda^i \, c_{k-i}(\alpha) =
$$
$$
c_k(\alpha) - (k-1)\lambda \, c_{k-1}(\alpha) + \frac{(k-1)(k-2)}{2} \lambda^2 \, c_{k-2}(\alpha) + \ldots +(-1)^{k-1} \lambda^{k-1} \, c_{1}(\alpha).
$$

\medskip

For a multi-index 
$$I = \{1 \le i_1 < \dots < i_r \le n\}$$ we let
$$|I| = i_1 + \dots + i_r$$
$$l(I) = r$$
$$c_I(\alpha) = c_{i_1}(\alpha) \cdot \dots \cdot c_{i_1}(\alpha) \in
H^{2|I|-l(I), |I|}(X).$$ 

\begin{proposition}
\label{MotiveGLn}
The motive $M(GL_n(F))$ admits the following direct sum decomposition:
$$
M(GL_n(F)) \cong 
\bigoplus_{I} \Tate (|I|)[2|I|-l(I)], 
$$
where the morphism
$$
M(GL_n(F)) \to \Tate (|I|)[2|I|-l(I)]
$$
corresponds to the class
$$
c_I(\alpha) \in H^{2|I|-l(I),|I|}(GL_n(F)),
$$
$[\alpha]$ is the tautological class in 
$K_1(GL_n(F))$ defined in the paragraph preceeding Proposition \ref{SplitGL1}.

\end{proposition}

\begin{proof}
We define the morphism
$$
\phi: M(GL_n(F)) \to \bigoplus_{I} \Tate (|I|)[2|I|-l(I)]
$$
using the classes $c_I$. We claim that $\phi$ is an isomoprhism.

First note, that for any reductive split group $G$ over $F$ the
motive $M(G)$ is a Tate motive.
This is deduced in Biglari's thesis \cite{Big} from the Bruhat decomposition
of $G$.  

Since $M(GL_n(F))$ is a Tate motive, by the Yoneda lemma it is sufficient to check that 
$\phi$ induces isomorphism on the motivic cohomology groups.

According to \cite{Pu}, Lemma 13, motivic cohomology of $GL_n(F)$ is generated freely
by the classes $c_I(\alpha)$
and the statement follows.

\end{proof}

\medskip

We also need the relative version of Proposition \ref{MotiveGLn}.

\begin{proposition}
\label{MotiveGLE}
Let $E \to T$ be a vector bundle of rank $n$, and let $\alpha_E$ be the
tautological class in $K_1(\GL(E))$.
The motive $M(\GL(E))$ admits the following decomposition:
$$
M(\GL(E)) =
\bigoplus_{I} M(T) (|I|)[2|I|-l(I)]
$$
where the morphism
$$
M(\GL(E)) \to M(T)(|I|)[2|I|-l(I)]
$$
is the composition
$$
M(\GL(E)) \to M(\GL(E)) \otimes M(\GL(E)) \to M(\GL(E)) (|I|)[2|I|-l(I)] \to
$$
$$
M(T)(|I|)[2|I|-l(I)]
$$
of multiplication by the class
$$
c_I(\alpha_E) \in H^{2|I|-l(I),|I|}(\GL(E)).
$$
followed by the canonical projection.
\end{proposition}

\begin{proof}
The statement follows from Proposition \ref{MotiveGLn} and the Mayer-Vietoris exact triangle.

\end{proof}

\subsection{The case $n=2$}

Let $A = \quat$, and let $C = SB(A)$.
In this case $GL_1(A)$ is the complement to $Q \subset \Aff^4 - \{ 0 \}$
in $\Aff^4 - \{ 0 \}$,
where
$$Q = \{(x,y,z,w) \in \Aff^4 - \{ 0 \} : x^2 - ay^2 - bz^2 + abw^2 = 0\}.$$

\begin{lemma}
$M(Q) = M(C) \oplus M(C)(2)[3]$. 
\end{lemma}

\begin{proof}
First note that the projective quadric $\{x^2 - ay^2 - bz^2 + abw^2 = 0\} \subset \Proj^3$ 
is isomorphic to $C \times C$.
It follows from Proposition \ref{SplitSB} that $C \times C$ is a projective line bundle over $C$,
therefore
$$M(C \times C) = M(C) \oplus M(C)(1)[2].$$

$Q$ over $C \times C$ is the complement to the zero section in the line bundle $\OO(-1)$.
We have a distinguished triangle

$$
M(C)(1)[1] \oplus M(C)(2)[3] \to M(Q) \to M(C) \oplus M(C)(1)[2] \to M(C)(1)[2] \oplus M(C)(2)[4],
$$  

and the third morphism is the obvious one and the claim follows.
  
\end{proof}

\begin{proposition}
\label{MGLquat}
There is a decomposition
$$
M(GL_1(A)) = \Tate \oplus M(C)(1)[1] \oplus \Tate_{a,b}(3)[4],
$$
where $\Tate_{a,b} = cone(\Tate(1)[2] \to M(C))$.
\end{proposition}

\begin{proof}
Consider the distinguished triangle corresponding to the open embedding $GL_1(A) \subset \Aff^4 - \{ 0 \}$:
$$
M(C)(1)[1] \oplus M(C)(3)[4] \to \widetilde{M}(GL_1(A)) \to \Tate(4)[7] \to M(C)(1)[2] \oplus M(C)(3)[5].
$$

By dimension reasons $Hom(\Tate(4)[7], M(C)(1)[2]) = 0$, therefore
$$\widetilde{M}(GL_1(A)) = M(C)(1)[1] \oplus cone(\Tate(4)[7] \to M(C)(3)[5])[-1].$$ 
The morphsim $\Tate(4)[7] \to M(C)(3)[5]$ corresponds to a class in $CH_1(C) = CH^0(C)$ which
can be computed after passing to a splitting field.
In the split case the morphism in question is the canonical morphism (the one
which corresponds to a cycle of degree 1), hence the same holds over $F$. 

\end{proof}

\begin{remark}
Note that in the split case $C = \Proj^1$ and $\Tate_{a,b} = \Tate$
so that the we have
$$
M(GL_2(F)) = \Tate \oplus \Tate(1)[1] \oplus \Tate(2)[3] \oplus \Tate(3)[4]
$$
in agreement with Proposition \ref{MotiveGLn}.
\end{remark}

\subsection{The general case}

We assume $n \ge 3$ is a prime.
Let $Z$ be the complement of $GL_1(A)$ in $\Affn$,
i.e. the subvariety in $\Affn$ given by equation $Nrd_A=0$.
Let $M = M_Z(\Affn)[-1]$, so that there is a distinguished triangle
$$
M \to M(GL_1(A)) \to M(\Affn) \to M[1].
$$

\begin{theorem}
\label{MotiveM}
1. For $j < n^2$ and $p \in \ZZ$ we have 
a canonical isomorphism
$$\widetilde{H}^{p,j}(GL_1(A)) \to H^{p,j}(M).$$  

2. If $A$ splits, then we have a decomposition
$$
M = \widetilde{M}(GL_1(A)) \oplus \Tate(n^2)[2n^2-2] =
\bigoplus_{I \ne \emptyset} \Tate_\XX(|I|)[2|I|-l(I)] \oplus
\Tate(n^2)[2n^2-2].
$$

3. $M$ is an object in $DT(\XX)$ and
the slices of the slice filtration are given by:
$$
\Piece{q}(M) = \left\{
\begin{array}{lll}

\bigoplus_{|I|=q} \Tate_\XX(q)[2q-l(I)] , & 1 \le q \le 
\frac{n(n+1)}{2}\\

\Tate_\XX(n^2)[2n^2-2], & q = n^2 \\

0, & otherwise
\end{array}
\right.
$$

\end{theorem}
\begin{proof}
Motivic cohomology of $GL_1(A)$ and that of $M$ are
related via the long exact sequence
$$
\widetilde{H}^{p,j}(\Affn) \to \widetilde{H}^{p,j}(GL_1(A)) \to H^{p,j}(M') \to \widetilde{H}^{p+1,j}(\Affn), 
$$
and the first claim follows since 
$$\widetilde{H}^{p,j}(\Affn) = H^{p,j}(\Tate(n^2)[2n^2-1]) = 0$$
for $j < n^2$ and any $p \in \ZZ$.

If the algebra $A$ is split, then in the distinguished triangle
$$
M \to \widetilde{M}(GL_n(F)) \to \widetilde{M}(\Affn) \to M[1]
$$
the second morphism is zero,
since as a simple computation using Proposition \ref{MotiveGLn}
shows, $Hom(\widetilde{M}(GL_n(F)), \widetilde{M}(\Affn))=0$.
The triangle splits yielding the first equality in the second claim.
The second equality follows from Proposition \ref{MotiveGLn}.

To prove the third claim note that any point of $z \in Z$ splits
$A$: $A_{F(z)}$ has a non-zero non-invertible element (given by $z$) 
therefore
$A_{F(z)}$ is not a division algebra, and since we assume that
the degree $n$ of $A$ is prime, $A_{F(z)}$ splits.
The third claim now follows from Lemmas \ref{XXInvar}, 2 
and \ref{TwistedTate}. 

\end{proof}

The slice filtration gives rise to a spectral sequence for each weight $j$
$$
E_2^{p,q} = H^{p+q}(\Piece{q}(M), \Tate(j)) \Rightarrow H^{p+q}(M, \Tate(j)).
$$

If we consider the weights $j < n^2$, then
by Theorem \ref{MotiveM}, 3
the spectral sequence in question actually converges to $\widetilde{H}^{*,j}(GL_1(A))$.

In the computation of the second differential in the spectral sequence
we will need the following isomorphism proved in \cite{MS}, Proposition 1.3:
$$
H^{3,1}(\Tate_\XX) \cong Ker(res: H^2_{et}(F, \mu_n) \to H^2_{et}(F(SB(A)), \mu_n)). 
$$
On the other hand for any field $H^2_{et}(F, \mu_l)$ is canonically isomorphic
to the $n$-torsion of the Brauer group $Br(E)$, and the kernel of the restriction
map $Br(F) \to Br(F(X))$ is generated by the class of algebra $A$ by the classical Amitsur theorem.
Since the period of $A$ is equal to $n$ we 
have a canonical isomorphism
$$
H^{3,1}(\Tate_\XX) \cong \ZZ/n.
$$

\begin{theorem}
\label{Differential}
Let $1 \le q \le \frac{n(n+1)}{2}$. The second differential $d_2$ in
the slice spectral sequence is induced by the morphism of motives
$$
\partial_q: \Piece{q}(M) = \bigoplus_{|I|=q} \Tate_\XX(q)[2q-l(I)] \to
\Piece{q+1}(M)[1] = \bigoplus_{|J|=q+1} \Tate_\XX(q+1)[2q+3-l(J)].
$$

The morphism 
$$\partial_{I,J}: \Tate_\XX(q)[2q-l(I)] \to \Tate_\XX(q+1)[2q+3-l(J)]$$
corresponding to multi-indices $I, |I|=q$ and $J, |J|=q+1$
is zero unless $l(I) = l(J)$ and the sequence $J$ is obtained
from the sequence $I$ by increasing one index $i_t$ by one, in
which case $\partial_{I,J}$ corresponds to the class in $H^{3,1}(\Tate_\XX)$
equal to $i_t \cdot c \cdot [A]$, for some integer $c$ coprime to $n$
depending only on $A$.
\end{theorem}

\begin{proof}
We fix a weight $q$ and a multi-index $I$ such that $|I|=q$,
and let $r = l(I)$.
Consider the motive $M(SB(A) \times GL_1(A))$. According to Proposition \ref{SplitGL1} 
$$SB(A) \times GL_1(A) = \GL_{SB(A)}(J^*),$$ and
Proposition \ref{MotiveGLE} implies that $M(SB(A) \times GL_1(A))$
admits a direct summand $M(SB(A))(q)[2q-r]$ corresponding to
the class $c_I(\alpha_E)$.
We claim that the composition $\psi$ defined as
$$
M(SB(A))(q)[2q-r] \to M(\GL_{SB(A)}(J^*)) = 
$$
$$
M(GL_1(A) \times SB(A)) \to M(GL_1(A))
$$
factors uniquely through $M \to M(GL_1(A))$. Indeed, from the distinguished
triangle defining $M$ we see that it is sufficient to show that
$$Hom(M(SB(A))(q)[2q-r],M(\Affn)[\epsilon])=0,$$
for $\epsilon=0,-1$.

$$Hom(M(SB(A))(q)[2q-r],M(\Affn)[\epsilon]) =$$
$$Hom(M(SB(A))(q)[2q-r],\Tate[\epsilon] \oplus \Tate(n^2)[2n^2-1+\epsilon]) =$$
$$H^{2n^2-1+\epsilon-(2q-r),n^2-q}(SB(A)) = 0,$$
since
$$2n^2-1+\epsilon-(2q-r)-(n^2-q)=n^2-q+r-1+\epsilon
> dim(SB(A))=n-1,$$
where we assume that $n \ge 3$ and can also assume that $q < \frac{n(n+1)}{2}$
since otherwise the statement of the theorem is trivial.

The morphism $$\phi: M(SB(A))(q)[2q-r] \rightarrow M$$ 
that we have just defined induces a morphism on the slice filtrations
of the source and target motives.
By Lemma \ref{TwistedTate} we have
$$
\Piece{k}(M(SB(A))(q)[2q-r]) = \Tate_\XX(k)[2k-r]
$$
for $q \le k \le n-1+q$.

On the other hand by Theorem \ref{MotiveM} we have
$$
\Piece{k}(M) = \bigoplus_{|J|=k} \Tate_\XX(k)[2k-l(J)].
$$

For each $J$ with $|J|=k$ we let $\Piece{k}(\phi) = \oplus \Piece{k}(\phi)_J$, 
$$ 
\Piece{k}(\phi)_J: \Tate_\XX(k)[2k-r] \to \Tate_\XX(k)[2k-l(J)].
$$ 

Each $\Piece{k}(\phi)_J$ is an element in 
$$Hom(\Tate_\XX(k)[2k-r]\, \Tate_\XX(k)[2k-l(J)]) =$$ 
$$Hom(\Tate_\XX, \Tate_\XX[r-l(J)]) = H^{r-l(J),0}(\Tate_\XX).$$
By Lemma \ref{CohomXX} the latter group is isomorphic to $\ZZ$
when $l(J)=r$ and is zero otherwise.

For the computation of the differential in the spectral sequence
we restrict our attention to the cases $k = q$, $k=q+1$.

\begin{lemma}
1. $\Piece{q}(\phi)_J = 1$ for $J = I$ and is zero if $J \ne I$.

2. $\Piece{q+1}(\phi)_J = i_t$ if the sequence $J$ is obtained
from $I$ by increasing an index $i_t$ by one, and is zero otherwise.  
\end{lemma}

Using the Lemma we finish the proof of the theorem. We have the
commutative diagram of the connecting morphisms in the slice filtrations: 

\[
\xymatrix{
\Tate_\XX(q)[2q-r] \ar[r]_-{\overline\partial} \ar[d]_{\Piece{q}(\phi)} & \Tate_\XX(q+1)[2q+3-r] \ar[d]_{\Piece{q+1}(\phi)} \\
\bigoplus_{|J|=q} \Tate_\XX(q)[2q-l(J)] \ar[r]_-{\partial_q} & \bigoplus_{|J|=q+1} \Tate_\XX(q+1)[2q+3-l(J)]
}
\]

From the first part of the Lemma it follows that the left
vertical map is the canonical embedding corresponding to $J = I$.
A diagram chase shows that
$$
\partial_{I,J} = \Piece{q+1}(\phi)_J \circ \overline\partial.
$$

It is proved by Suslin in \cite{Sus} that $\overline \partial \in H^{3,1}(\Tate_\XX)$
is equal to $c [A]$, for $c$ coprime to $n$. 
By the second part of the Lemma, we obtain the description of
the differential.  

\medskip

It remains to prove the Lemma above.
According to Lemma \ref{CohomXX}, integers
$\Piece{q}(\phi)_J$ and $\Piece{q+1}(\phi)_J$
do not change under the extension of scalars to the field $F(SB(A))$.
Therefore we may assume that $A$ is split.

In this case for any $1 \le q \le \frac{n(n+1)}{2}$ we have 
$\Piece{q}(M) = \Piece{q}(GL_n(F))$ by Theorem \ref{MotiveM}, 2
and thus $\Piece{q}(\phi)$ is identified with $\Piece{q}(\psi)$.

The morphism $\psi$ has the form:
$$
\psi: M(\Proj(V))(q)[2q-r] \to M(\GL_{\Proj(V)}(J^*)) = M(\Proj(V) \times GL_n(F)) \to M(GL_n(F)).
$$

For each $q \le k \le q+n-1$ we have
$$\Piece{k}(\psi): \Tate(k)[2k-r] \to \bigoplus_{|J|=k} \Tate(k)[2k-l(J)].$$ 
$\Piece{k}(\psi)_J$ can be non-zero only for $J$ with $l(J) = r$,
and for such $J$, we have
$$
\Piece{k}(\psi)_J = \psi^*(c_J(\alpha_0)) \in CH^{k-q}(\Proj(V)).
$$
$\psi^*$ above is the morphism induced on motivic cohomology:
$$
\psi^*: H^{*,*}(GL_n(F)) \to H^{*-(2q-r)),*-q}(\Proj(V)).
$$

By Proposition \ref{MotiveGLE} motivic cohomology $H^{*,*}(\GL_{\Proj(V)}(J^*))$ is a free module over $H^{*,*}(\Proj(V))$
with a basis $c_J(\alpha_{J^*})$.
We also have another basis $c_J(p_2^* \alpha_0)$,
and the map $\psi^*$ above can be described in terms of these bases
as follows: $\psi^*(c_J(\alpha_0))$ is the coefficient of 
$\lambda^{|J|-q} c_I(\alpha_{J^*})$
in $c_J(p_2^* \alpha_0)$, where $\lambda = [\OO(1)] \in CH^1(\Proj(V))$.

By Proposition \ref{SplitGL1} 
$[p_2^*(\alpha_0)] = [p_1^*(\OO(-1))] \cdot [\alpha]$
and by the properties of the higher Chern classes listed before Proposition
\ref{MotiveGLn} we compute:

$$
c_{j_1, \dots, j_r} (p_2^*(\alpha_0)) = \prod_{t=1}^r c_{j_t}(p_2^*(\alpha_0)) =
$$
$$
\prod_{t=1}^r (c_{j_t}(\alpha) + (j_t-1)\lambda \, c_{j_t-1}(\alpha) + \dots) =
$$
$$ 
c_{j_1, \dots, j_r}(\alpha) + \sum_{t=1}^r (j_t-1)\lambda \, c_{j_1,\dots,j_t-1, \dots, j_r} (\alpha) + \dots 
$$

Taking $|J|=q$, $|J|=q+1$ in the formula above finishes the proof.

\end{proof}

\begin{corollary} 
\label{SmallWeights}
The extension of scalars to a splitting field of $A$ 
identifies the weight 1 and 2 motivic cohomology of $GL_1(A)$ with:
$$
\widetilde{H}^{p,1}(GL_1(A)) = \left\{
\begin{array}{lll}
\ZZ         , & p=1 \\
0           , & otherwise
\end{array}
\right.
$$

$$
\widetilde{H}^{p,2}(GL_1(A)) = \left\{
\begin{array}{lll}
F^*         , & p=2 \\
n \ZZ        , & p=3 \\
0           , & otherwise
\end{array}
\right.
$$

$$
\widetilde{H}^{p,3}(GL_1(A)) = \left\{
\begin{array}{lll}
H^{0,2}(F)   , & p=1 \\
H^{1,2}(F)   , & p=2 \\
H^{2,2}(F)   , & p=3 \\
\ZZ \oplus (F^*)^n ,  & p = 4 \\
n \ZZ        , & p = 5 \\
0, & otherwise
\end{array}
\right.
$$


\end{corollary}

\begin{proof}
In weight $j$ the spectral sequence has nonzero terms 
$$E_2^{p,q}=H^{p+q}(\Piece{q}(M), \Tate(j))$$ 
only for $0 < q \le j$.
Let us now consider the weights $j=1,2,3$. In these weights the spectral
sequence converges to $\widetilde{H}^{*,j}(GL_1(A)$ by theorem \ref{MotiveM}, 1.   
The first three slices of the slice filtration are given by:
\begin{align*}
\Piece{1}(M) = & \Tate_\XX(1)[1] \\
\Piece{2}(M) = & \Tate_\XX(2)[3] \\
\Piece{3}(M) = & \Tate_\XX(3)[4] \oplus \Tate_\XX(3)[5].
\end{align*}

We need the following result on motivic cohomology of $\XX$:
$$
H^{p,0}(\XX) = \left\{
\begin{array}{lll}
\Tate       , & p=0 \\
0           , & otherwise
\end{array}
\right.
$$
which follows from Lemma \ref{CohomXX},
and

$$
H^{p,1}(\XX) = \left\{
\begin{array}{lll}
F^*         , & p=1 \\
\ZZ/n \cdot [A]   , & p=3 \\
0           , & otherwise
\end{array}
\right.
$$

$$
H^{p,2}(\XX) = \left\{
\begin{array}{lll}
H^{p,2}(F)  , & p \le 2 \\
K_1(F)/n \cdot [A] , & p=4 \\
0           , & otherwise
\end{array}
\right.
$$
These formulas follow easily from the results on
motivic cohomology of more general simplicial 
schemes $\XX_\theta$ proved in \cite{MS}.


In the weight $j=1$ the slice spectral sequence consists of one row which contains
a unique non-zero term $E_2^{0,1} = H^{0,0}(\XX) = \ZZ$, hence we get 
the isomorphism
$$
\widetilde{H}^{1,1}(GL_1(A)) = \Tate
$$ 
and the other reduced cohomology groups of $GL_1(A)$ of weight 1 vanish. 
 
In the weight $j=2$ we have two nonzero rows:
\begin{align*}
E_2^{p,1}= H^{p+1,2}(\XX(1)[1])= & \, H^{p,1}(\XX) \\
E_2^{p,2}= H^{p+2,2}(\XX(2)[3])= & \, H^{p-1,0}(\XX)
\end{align*}

\[
\xymatrix{
0 & \ZZ \ar[drr]_{d^2} & 0 & 0 \\
0 & F^* & 0 & \ZZ/n \\
0 & 0 & 0 & 0 
}
\]
and the differential $d_2$ is multiplication by $c$ which is coprime to $n$,
thus
$$
\widetilde{H}^{2,2}(GL_1(A)) = F^*
$$
$$
\widetilde{H}^{3,2}(GL_1(A)) = n \ZZ
$$
and the other reduced cohomology groups of $GL_1(A)$ of weight 2 vanish. 

In the weight $j=3$ we have three nonzero rows:
$$
E_2^{p,1}= H^{p+1,3}(\XX(1)[1])= H^{p,2}(\XX) 
$$
$$
E_2^{p,2}= H^{p+2,3}(\XX(2)[3])= H^{p-1,1}(\XX) 
$$
$$
E_2^{p,3}= H^{p+3,3}(\XX(3)[4]\oplus\XX(3)[5])= H^{p-1,0}(\XX) \oplus H^{p-2,0}(\XX).
$$

\[
\xymatrix{
0 & \ZZ & \ZZ \ar[drr]^{d^2}  & 0 & 0\\
0 & 0 & F^* \ar[drr]^{d^2} & 0 & H^{3,1}(\XX)\\
H^{0,2}(F) & H^{1,2}(F) & H^{2,2}(F) & 0 & H^{4,2}(\XX)\\
0 & 0 & 0 & 0 & 0
}
\]
Both non-zero differentials are surjective with kernels consisting
of elements divisible by $n$. There are no higher differentials by
dimension reasons.

\end{proof}

\section{The case of $SL_1(A)$}

In this section we investigate the motive of $SL_1(A)$ in the case of a division algebra
$A$ of degree 2 or 3 by looking at the smooth compactification of $SL_1(A)$. The compactification in
question arises as a hyperplane section of the generalized Severi-Brauer variety associated to $M_2(A)$.

\subsection{Hyperplane sections of the generalized Severi-Brauer varieties}

\medskip

Let $A$ be a division algebra of degree $n$ and
let $X_A$ be the closed subvariety
of $SB_n(M_2(A))$ given by equation
$$Nrd(\alpha_1) = Nrd(\alpha_2).$$
$X_A$ is a hyperplane section of $SB_n(M_2(A))$ with respect to
the Pl\"{u}cker embedding \ref{Plucker}.

On the open dense subscheme of $X_A$ where $Nrd(\alpha_1) \ne 0$ we have
$$[\alpha_1, \alpha_2] = [1, \alpha_1^{-1}\alpha_2] = [1, \alpha],$$
with the condition $Nrd(\alpha) = 1$. Hence this open subscheme is isomorphic to $SL_1(A)$.

\begin{proposition}
\label{Compactif}
$X$ is smooth if the degree $n = 2$ or $3$.
\end{proposition}

\begin{proof}
We concentrate on the case $n=3$, the case $n=2$ being similar.
It suffices to consider the case when $A$ splits.
In this case, according to \ref{Istar}, $SB_3(M_2(A))$ is isomorphic to $Gr(3, 6)$,
and $Nrd$ is identified with $det$.
We identify points of $Gr(3,6)$ with $6 \times 3$ matrices of rank 3 modulo $GL_3(F)$ action on the columns.
We cover $Gr(3,6)$ by open charts where the three given rows are linearly independent.

$X_A$ intersects the open chart where the first three (or the last three) rows are linearly independent,
in a variety isomorphic to $SL_3(F)$, which is obviously smooth.

All other cases will look essentially like this:

\[ \left( \begin{array}{ccc}
1 & 0 & 0\\
0 & 1 & 0\\
a_{31} & a_{32} & a_{33}\\
\hline
0 & 0 & 1\\
a_{51} & a_{52} & a_{53}\\
a_{61} & a_{62} & a_{63}\end{array} \right),\]

in which case the equation defining $X_A$ becomes
$$
a_{33} = a_{51}a_{62} - a_{52}a_{61},
$$
so that in these charts $X_A$ is also smooth.
\end{proof}

\begin{remark}
It is clear from the proof of the proposition that $X_A$ will not be smooth
if $n > 3$. Indeed, in this case the equation defining $X_A$ in some of the charts will
be polynomials that do not contain constant or linear terms in $a_{ij}$.
\end{remark}

\subsection{The case $n=2$}

Let $A = \quat$ be the quaternion algebra, and $X = X_A$.
As indicated in Remark \ref{Plucker} we can pick coordinates such that the Pl\"{u}cker embedding
$SB_2(M_2(A)) \to \Proj^5$
sends $[\alpha_1, \alpha_2]$ to
$$[Nrd(\alpha_1), Nrd(\alpha_2), Nrd(\alpha_1+\alpha_2), Nrd(\alpha_1+i\alpha_2),
Nrd(\alpha_1+j\alpha_2),
Nrd(\alpha_1+k\alpha_2)].$$

Note that $SB_2(M_2(A))$ is a hypersurface in $\Proj^5$.
We can change a coordinate system, so that the equation defining $SB_2(M_2(A))$ becomes
particularly simple.

\begin{proposition}
Let $$t_1 = Nrd(\alpha_1)$$
$$t_2 = Nrd(\alpha_2)$$ and
$$u_1 = \frac{1}{2}(Nrd(\alpha_1 + \alpha_2) - Nrd(\alpha_1) - Nrd(\alpha_2))$$
$$u_2 = -\frac{1}{2a}(Nrd(\alpha_1 + i\alpha_2) - Nrd(\alpha_1) - Nrd(i\alpha_2))$$
$$u_3 = -\frac{1}{2b}(Nrd(\alpha_1 + j\alpha_2) - Nrd(\alpha_1) - Nrd(j\alpha_2))$$
$$u_4 = \frac{1}{2ab}(Nrd(\alpha_1 + k\alpha_2) - Nrd(\alpha_1) - Nrd(k\alpha_2)).$$
Then $SB_2(M_2(A))$ is defined by the equation
$$
t_1 t_2 = u_1^2 - a u_2^2 - b u_3^2 + ab u_4^2.
$$
\end{proposition}

\begin{proof}
First note that the identity
$$
Nrd(\alpha+\beta) = Nrd(\alpha) + Nrd(\beta) + Tr(\alpha \overline \beta).
$$
implies that
$$u_1 = \frac{1}{2}Tr(\alpha_1 \overline \alpha_2)$$
$$u_2 = -\frac{1}{2a}Tr(\alpha_1 \overline{i \alpha_2})= \frac{1}{2a}Tr(\alpha_1 \overline \alpha_2 i)$$
$$u_3 = -\frac{1}{2b}Tr(\alpha_1 \overline{j \alpha_2})= \frac{1}{2b}Tr(\alpha_1 \overline \alpha_2 j)$$
$$u_4 = \frac{1}{2ab}Tr(\alpha_1 \overline{k \alpha_2})= -\frac{1}{2ab}Tr(\alpha_1 \overline \alpha_2 k).$$

Now,
$$
\alpha_1 \overline\alpha_2 =
\frac{1}{2}Tr(\alpha_1 \overline\alpha_2) +
\frac{1}{2a}Tr(\alpha_1 \overline\alpha_2 i)i +
\frac{1}{2b}Tr(\alpha_1 \overline\alpha_2 j)j -
\frac{1}{2ab}Tr(\alpha_1 \overline\alpha_2 k)k =
$$
$$
u_1 + u_2 i + u_3 j + u_4 k,
$$
so that
$$
t_1 t_2 = Nrd(\alpha_1) Nrd(\alpha_2) = Nrd(\alpha_1) Nrd(\overline \alpha_2) =
$$
$$
Nrd(\alpha_1 \overline \alpha_2) = Nrd(u_1 + u_2 i + u_3 j + u_4 k) = u_1^2 - a u_2^2 - b u_3^2 + ab u_4^2.
$$

It follows that $SB_2(M_2(A))$ is contained in the quadric
$$t_1 t_2 = u_1^2 - a u_2^2 - b u_3^2 + ab u_4^2.$$
Since both varieties are irreducible and of the same dimension,
they must coincide.

\end{proof}

Let $C = SB(A)$.

\begin{proposition} 
\label{MSLquat}
$M(SL_1(A)) = \Tate \oplus \Tate_{a,b}(2)[3]$, 
where $\Tate_{a,b} = cone(\Tate(1)[2] \to M(C))$.
\end{proposition}

\begin{proof}
The equation defining $X$ in $\Proj^4$ is
$$
t^2 = u_1^2 - a u_2^2 - b u_3^2 + ab u_4^2,
$$
which is the projective closure of the affine quadric $SL_1(A) \subset \Aff^4$
$$
u_1^2 - a u_2^2 - b u_3^2 + ab u_4^2 = 1.
$$

Quadratic form $<1, -a, -b, ab, -1>$ is equivalent to
the sum of a hyperbolic plane and $<1, -a, -b>$.
It follows from the work of Rost \cite{R} that 
$$
M(X) = \Tate \oplus M(C)(1)[2] \oplus \Tate(3)[6]. 
$$

The complement to $SL_1(A)$ inside $X$ is the smooth Pfister quadric
$$
u_1^2 - a u_2^2 - b u_3^2 + ab u_4^2 = 0,
$$
isomorphic to $C \times C$.

We have the localization distinguished triangle

$$
\widetilde{M}(SL_1(A)) \to M(C)(1)[2] \oplus \Tate(3)[6] \to M(C)(1)[2] \oplus M(C)(2)[4] \to M(SL_1(A))[1].
$$

One can see that 
$$\widetilde{M}(SL_1(A)) = cone(\Tate(3)[6] \to M(C)(2)[4])[-1] =$$ 
$$cone(\Tate(1)[2] \to M(C))(2)[3]) = \Tate_{a,b}(2)[3].$$

\end{proof}

\subsection{The case $n=3$}

In this section we consider the motive of $X_A$ for an algebra of degree 3.

Let $T$ be a smooth projective variety of dimension $d$ over $F$.
Consider a Tate motive
$\bigoplus_{i=0}^d \Tatepure{i}^{\oplus k_i}$,
which we require to be $d$-self-dual: $k_i = k_{d-i}$ for all $i$.
Let $\phi$ be a morphism of motives:
$$
\phi: \bigoplus_{i=0}^d \Tatepure{i}^{\oplus k_i} \to M(T).
$$

We can consider the dual morphism
$$
M(T)(-d)[-2d] = M(T)^* \to \bigoplus_{i=0}^d \Tatepure{-i}^{\oplus k_i},
$$
and its twist
$$
\phi^t: M(T) \to \bigoplus_{i=0}^d \Tatepure{d-i}^{\oplus k_i} =
\bigoplus_{i=0}^d \Tatepure{i}^{\oplus k_i}.
$$

\begin{remark}
We have canonical identifications:
$$
CH_i(T) = Hom_{\DMeff}(\Tatepure{i}, M(T))
$$
$$
CH^i(T) = Hom_{\DMeff}(M(T), \Tatepure{i}).
$$
From this point of view both $\phi$ and $\phi^t$
correspond to the same set of elements
$\{ \alpha_{ij} \in CH_i(T)=CH^{d-i}(T), j=1 \dots k_i \}_{i=0}^d$

\end{remark}

\begin{lemma}
\label{TateIso}
The composition $\phi^t \circ \phi$ is an isomorphism
if and only if the Gram matrix for the intersection product of $\{a_{ij}\}$ has
an invertible determinant.
\end{lemma}

\begin{proof}
First recall, that
$$Hom(\Tatepure{n}, \Tatepure{m}) =
\left\{
\begin{array}{lll}
\ZZ         , & n=m \\
0           , & otherwise
\end{array}
\right.$$

\newcommand{\tatepair}[0]{\Tatepure{i}^{\oplus k} \oplus \Tatepure{d-i}^{\oplus k}}
\newcommand{\tatehalf}[0]{\Tatepure{d/2}^{\oplus k}}

This reduces the general case to the following two special cases
of the $d$-self-dual Tate motive: $\tatepair$, $i \ne d/2$
and $\tatehalf$ (if $d$ is even).
In the first case, $\phi: \tatepair \to M(T)$ is determined by the classes
$\alpha_1, \dots, \alpha_k \in CH_i(T)$ and $\beta_1, \dots
\beta_k \in CH^i(T)$. Let $A$ be the $k \times k$ matrix with entries $\alpha_i \cdot \beta_j \in \ZZ$.
Then the matrix of $\phi^t \circ \phi$ is the $2k \times 2k$ block matrix
\[ \left( \begin{array}{cc}
0 & A \\
A^t & 0
\end{array} \right),\]
which is invertible if and only $A$ is invertible.

The case of $\tatehalf$ is similar: the morphism $\phi: \tatehalf \to M(T)$
is determined by the collection of classes $\alpha_1, \dots, \alpha_k \in CH^{d/2}(T)$,
and the matrix of $\phi^t \circ \phi$ \it is equal to the Gram matrix of $\{\alpha_i\}$.

\end{proof}

\label{SchubertCalc}

Let $A$ be an algebra of degree 3 over a field $F$ of characteristic 0.
We set $X = X_A$.
We start with the split case $A = End(V)$ when $SB_3(M_2(A))$ 
is identified with $Gr(3, V \oplus V) = Gr(3,6)$.

Recall that $Gr(3,6)$ is a cellular variety and the Chow groups
$CH^*(Gr(3,6))$ are freely generated by Schubert cells $\Delta_\lambda$
corresponding to Young diagrams $\lambda$ with 3 rows and 3 columns \cite{F}.
We will be using the so-called Pieri formula \cite{F} to compute the product of any $\Delta_\lambda$ with the generator
$\Delta_{(1)} = \tiny\yng(1) \in CH^1(Gr(3,6))$.
The result of multiplication is the sum of all $\Delta_{\lambda'}$
for partitions $\lambda'$ which can be obtained from $\lambda$ by adding one box.
For example, $$\tiny\yng(1) \cdot \tiny\yng(2,1,1) = \tiny\yng(2,2,1) + \tiny\yng(3,1,1).$$

Following Semenov's paper \cite{Sem}, in the split case $\Gm$ is acting on $X$ with finitely many fixed points,
and using the results of Bialinicki-Birula \cite{B} we conclude that in split case $X$ is also cellular, in particular the motive of $X$ is a Tate motive and all cohomology classes on $X$ are algebraic.
It follows now from the Weak Lefschetz theorem that the natural pull-back and push-forward maps
$$CH^i(Gr(3,6)) \to CH^i(X), \quad 0 \le i \le 3$$
and
$$CH^i(X) \to CH^{i+1}(Gr(3,6)), \quad 5 \le i \le 8$$
are isomorphisms.

The map in the middle codimension
$$CH^4(Gr(3,6)) \to CH^4(X)$$
is injective with cokernel of rank 1.

Under these identifications we will consider Schubert classes corresponding to partitions $\lambda$ with
$|\lambda| = i$
as elements of $CH^i(X)$ for $i \le 4$
and as elements of $CH^{i-1}(X)$ for $i \ge 4$.
For example
$$\tiny\yng(2,1,1)_X := i^*(\tiny\yng(2,1,1)_{Gr}) \in CH^4(X)$$ whereas
$$\tiny\yng(3,3)_X := i_*^{-1}(\tiny\yng(3,3)_{Gr}) \in CH^5(X).$$

For these classes the Pieri formula holds as above with the exception of codimension 4 case,
where one has to apply it twice, for example
$$
\tiny\yng(1)_X \cdot \tiny\yng(2,2)_X = i_*^{-1} i_*(\tiny\yng(1)_{X} \cdot \tiny\yng(2,2)_{X}) =
i_*^{-1} i_* i^*(\tiny\yng(1)_{Gr} \cdot \tiny\yng(2,2)_{Gr}) =
$$
$$
i_*^{-1}(\tiny\yng(1)_{Gr}^2 \cdot \tiny\yng(2,2)_{Gr}) = i_*^{-1}(\tiny\yng(3,3)_{Gr}+2\tiny\yng(3,2,1)_{Gr}+\tiny\yng(2,2,2)_{Gr}) =
\tiny\yng(3,3)_{X}+2\tiny\yng(3,2,1)_{X}+\tiny\yng(2,2,2)_{X}
$$




\medskip

Now consider the general (non-split) case.
Let $E$ be a splitting field of the algebra $A/F$. For a variety
$T/F$ we call a class $\alpha \in CH^*(T_E)$ {\it rational} if it
lies in the image of the extension of scalars map $CH^*(T) \to
CH^*(T_E)$.

\medskip

\begin{proposition}
\label{RationalClasses}
The following classes are rational:

1. $\tiny\yng(1) \in CH^1(Gr(3,6)_E), \tiny\yng(2,2) \in CH^4(Gr(3,6)_E)$

2. $\tiny\yng(3) + H \tiny\yng(2) + H^2 \tiny\yng(1) \in CH^3(\mathbb P^2_E
\times Gr(3,6)_E)$, where $H$ is the hyperplane section class on $\mathbb P^2_E$.

3. All classes divisible by 3 in $CH^*(\mathbb P^2_E),
CH^*(Gr(3,6)_E)$ or $CH^*(X_E)$
\end{proposition}

\begin{proof}

1. For $i = 1, 2$ consider the subvarieties $Z_i \subset SB_3(M_2(A))$ defined by equations
$$
rk(\alpha_i) \le i.
$$
In fact, extending scalars to the splitting field $E$,
$Z_1$ and $Z_2$ are the closures of the open Schubert cells

\[ \left( \begin{array}{ccc}
1 & 0 & 0\\
* & 0 & 0\\
* & 0 & 0\\
\hline
0 & 1 & 0\\
0 & 0 & 1\\
* & * & *\end{array} \right)\]

and

\[ \left( \begin{array}{ccc}
1 & 0 & 0\\
0 & 1 & 0\\
* & * & 0\\
\hline
0 & 0 & 1\\
* & * & *\\
* & * & *\end{array} \right)\]

respectively. Therefore the class of $Z_1$ is $\tiny\yng(2,2)$,
and the class of $Z_2$ is $\tiny\yng(1)$.

2. According to Proposition \ref{SplitSB} the bundle
$p_1^*(\OO(1)) \otimes p_2^*(\QQ)$ is rational.
For any rank 3 bundle $E$ and line bundle $L$
$$c_3(L \otimes E) = c_3(E) + c_1(L) c_2(E) + c_1(L)^2 c_1(E),$$
therefore $$c_3(p_1^*(\OO(1)) \otimes p_2^*(\QQ)) = \tiny\yng(3) + H \tiny\yng(2) + H^2 \tiny\yng(1)$$
and thus this cycle is rational.

3. This follows from the transfer argument, since $A$ has a splitting field of degree 3.

\end{proof}

\begin{proposition}
\label{MotDecomp}
Consider the following five elements
$$\alpha_1 = \tiny\yng(3) + H \tiny\yng(2) + H^2 \tiny\yng(1) \in CH^3(\mathbb P^2_E \times X_E)$$
$$\alpha_2 = \tiny\yng(3,1) + H (\tiny\yng(3) + \tiny\yng(2,1)) + H^2 (\tiny\yng(2) + \tiny\yng(1,1)) \in CH^4(\mathbb P^2_E \times X_E)$$
$$\alpha_3 = (\tiny\yng(3,3) - \tiny\yng(3,2,1)) + H (-\tiny\yng(3,1) + \tiny\yng(2,2) + \tiny\yng(2,1,1)) +
H^2 (\tiny\yng(3) - \tiny\yng(2,1) + \tiny\yng(1,1,1)) \in CH^5(\mathbb P^2_E
\times X_E)$$
$$\alpha_4 = -\tiny\yng(3,2,2) + H (-\tiny\yng(3,2,1) -\tiny\yng(2,2,2)) + H^2 (-\tiny\yng(2,2)) \in CH^6(\mathbb P^2_E \times X_E)$$
$$\alpha_5 = -\tiny\yng(3,3,2) + H (-\tiny\yng(3,3,1) + \tiny\yng(3,2,2)) +
H^2 (-\tiny\yng(3,3) + \tiny\yng(3,2,1) - \tiny\yng(2,2,2)) \in CH^7(\mathbb
P^2_E \times X_E)$$

Then

1. $\alpha_1, \dots, \alpha_5$ are rational cycles, therefore each
$\alpha_i$ defines a morphism
$$\MSBD{6-i} \to M(X)$$

2. Cycles $\alpha_i, i = 1 \dots 5$ together with the canonical
maps $\Tate \to M(X)$ and $\Tatepure{8} \to M(X)$ define a
morphism
$$\phi: \mathbb Z \oplus \bigoplus_{i=1}^5 \MSBD{i} \oplus \mathbb Z\{8\} \to M(X)$$
which is an embedding of a direct summand if we consider a coefficient ring where $2$
is invertible. The complementary direct summand is a form
of $\mathbb Z\{4\}$
\end{proposition}

\begin{proof}
The class $\alpha_1$ is rational by Proposition \ref{RationalClasses}.
Now using Pieri formula from section \ref{SchubertCalc} one can check that for all $i$,
$$\alpha_{i+1} \equiv \tiny\yng(1) \cdot \alpha_{i} \quad mod \quad 3,$$
which again by Proposition \ref{RationalClasses} implies that all $\alpha_i$ are rational.

To prove the second claim, consider the dual map
$$\phi^t: M(X) \to \Tate \oplus \bigoplus_{i=1}^5 \MSBD{i} \oplus \Tatepure{8}.$$

It is sufficient to check that the composition $\phi^t \circ \phi$ is an isomorphism.
First note that by dimension reasons there is no morphisms between $\Tate \oplus \Tatepure{8}$ and
$\bigoplus_{i=1}^5 \MSBD{i}$. Therefore what we have to check is that the restrictions of $f^t \circ f$
to $\Tate \oplus \Tatepure{8}$ and $\bigoplus_{i=1}^5 M(SB(A))(i)$ are isomorphisms.

The restriction of $\phi^t \circ \phi$ to $\Tate \oplus \Tatepure{8}$ is an isomorphism obviously (or by Lemma \ref{TateIso}).

The restriction of $\phi^t \circ \phi$ to $\bigoplus_{i=1}^5 \MSBD{i}$ is an isomorphism if it is isomorphism in split case by Proposition \ref{RostNilp}.
In split case, using Lemma \ref{TateIso}, it is sufficient to check that the intersection pairing, restricted to
the subspace of $CH^*(X)$ generated by all Schubert classes contained in $\alpha_1, \dots, \alpha_5$
is nondegenerate. It is the case in codimension $\ne 4$, since the classes

$$\tiny\yng(1)$$

$$\tiny\yng(2), \tiny\yng(2) + \tiny\yng(1,1)$$

$$\tiny\yng(3), \tiny\yng(3) + \tiny\yng(2,1), \tiny\yng(3) - \tiny\yng(2,1) + \tiny\yng(1,1,1)$$

$$\tiny\yng(3,3) - \tiny\yng(3,2,1), -\tiny\yng(3,2,1) -\tiny\yng(2,2,2), -\tiny\yng(3,3) + \tiny\yng(3,2,1) - \tiny\yng(2,2,2)$$

$$-\tiny\yng(3,2,2), -\tiny\yng(3,3,1) + \tiny\yng(3,2,2)$$

$$-\tiny\yng(3,3,2)$$

form the bases of the groups $CH^i(X)$ of the respective codimensions $1, 2, 3, 5, 6, 7$.

As for codimension $4$, the elements

$$\tiny\yng(3,1), -\tiny\yng(3,1) + \tiny\yng(2,2) + \tiny\yng(2,1,1), -\tiny\yng(2,2)$$

span the subspace of $CH^4(X)$ consisting of cycles which pull-back from $Gr(3,6)$.

The Gram matrix for the elements $\tiny\yng(3,1), \tiny\yng(2,2), \tiny\yng(2,1,1)$ is

\[ \left( \begin{array}{ccc}
1 & 1 & 0 \\
1 & 0 & 1 \\
0 & 1 & 1 \end{array} \right)\]

This matrix has determinant $-2$, which implies that the conditions of Lemma \ref{TateIso} hold if $2$ is invertible
in the coefficient ring.

\end{proof}

\begin{remark}
The complementary summand in the decomposition from Proposition \ref{MotDecomp} corresponds to the vanishing cycle for the embedding of $X$ into the form of $SB_2(M_3(A))$. Indeed, if $A$ splits, the middle degree group $CH^4(X)$ is free of rank 4, with 3 generators {\tiny\yng(3,1), \tiny\yng(2,2), \tiny\yng(2,1,1)} coming from $Gr(3,6)$. The fourth generator is the vanishing cycle $\delta$ as
in Picard-Lefschetz theory. 
This cycle is characterized by the property that it is sent to $0$ under the direct image map.
One can prove that if $A$ splits, there exists a cycle $Z$ in the class of {\tiny\yng(2,2)} on $Gr(3,6)$ 
such that the intersection of $Z$ with $X$ consists of two components $\alpha$ and $\beta$ with
multiplicities one,
and the vanishing cycle is expressed as
$$\delta = \alpha-\beta+{\tiny\yng(3,1)}-{\tiny\yng(2,1,1)}$$

When $A$ is not split, {\tiny\yng(2,2)} is still a rational class.
However, both $\alpha$ and $\beta$ as well as {\tiny\yng(3,1)} and {\tiny\yng(2,1,1)} are not rational (even as elements of Chow groups), therefore it is not clear how to describe $\delta$ in general.
\end{remark}

\end{document}